\documentclass[lineno]{biometrika}
\usepackage{setspace}
\usepackage{amssymb,caption,fancyhdr}
\usepackage{natbib}
\usepackage[plain,noend]{algorithm2e}

\usepackage{graphicx,amscd,amsmath,amssymb,amsfonts,verbatim}
\usepackage{mathrsfs,graphics}
\usepackage{bbm}
\usepackage[all]{xy}
\usepackage{xcolor}
\usepackage{bm}
\usepackage{bbm}

\def\E{\mathbb{E}}
\def\E{\mathbbmss{E}}
\def\P{\mathbbmss{P}}

\def\Normal{\mathcal{N}}
\def\I{\mathbbmss{1}}

\def\bx{\mathbf{x}}

\def\bX{\mathbf{X}}

\def\bx{\mathbf{x}}

\def\bX{\mathbf{X}}

\def\Tpop{T^{\text{pop}}}

\def\Tpop{T^{\text{pop}}}

\def\bSig\mathbf{\Sigma}

\renewenvironment{proof}[1][Proof]{\begin{trivlist}
\item[\hskip \labelsep {\bfseries #1}]}{\end{trivlist}}

\renewcommand{\qed}{\nobreak \ifvmode \relax \else
      \ifdim\lastskip<1.5em \hskip-\lastskip
      \hskip1.5em plus0em minus0.5em \fi \nobreak
      \vrule height0.75em width0.5em depth0.25em\fi}

\def\Tpop{T^{\text{pop}}}


\begin{document}

\title{The Propensity Score Estimation in the Presence of Length-biased Sampling: A Nonparametric Adjustment Approach}
%  This will produce the submission and review information that appears
%  right after the reference section.  Of course, it will be unknown when
%  you submit your paper, so you can either leave this out or put in
%  sample dates (these will have no effect on the fate of your paper in the
%  review process!)
\markboth{ A. Ertefaie, M. Asgharian and D. Stephenes }{Miscellanea}

\author{Ashkan Ertefaie}
\affil{Department of Statistics, University of Michigan, Ann Arbor, Michigan, USA \email{ertefaie@umich.edu} }

\author{Masoud Asgharian \and David Stephens}
\affil{Department of Mathematics and Statistics, McGill University, Montreal, Quebec, Canada. \email{masoud@math.mcgill.ca}  \email{dstephens@math.mcgill.ca}}

\maketitle

%  These options will count the number of pages and provide volume
%  and date information in the upper left hand corner of the top of the
%  first page as in published papers.  The \pagerange command will only
%  work if you place the command \label{firstpage} near the beginning
%  of the document and \label{lastpage} at the end of the document, as we
%  have done in this template.

%  Again, putting a volume number and date is for your own amusement and
%  has no bearing on what actually happens to your paper!

%  The \doi command is where the DOI for your paper would be placed should it
%  be published.  Again, if you make one up and stick it here, it means
%  nothing!

%%\doi{10.1111/j.1541-0420.2005.00454.x}

%  This label and the label ``lastpage'' are used by the \pagerange
%  command above to give the page range for the article.  You may have
%  to process the document twice to get this to match up with what you
%  expect.  When using the referee option, this will not count the pages
%  with tables and figures.

\label{firstpage}

%  put the summary for your paper here

\begin{abstract}

%We consider estimating the propensity score  from observational data when the data constitute a length-biased sample from the target population. Length-bias in survival data occurs in observational studies when, for example, subjects with shorter lifetimes are less likely to be present in the recorded data. Thus, depending on the association between covariates and survival time, some covariates will be under or overrepresented in the observed sample. \cite{cheng2012estimating} introduced a method that adjusts for this bias which requires the correct conditional survival/hazard function given the treatment and covariates. We introduce a nonparametric adjustment technique  based on a weighted estimating equation for estimating the propensity score which does not require any modeling assumption for the conditional survival function. Large sample properties of the estimator is established and its small sample behaviour is studied using simulations. 

The pervasive use of prevalent cohort studies on disease duration, increasingly calls for appropriate methodologies to account for the biases that invariably accompany samples formed by such data. It is well-known, for example, that subjects with shorter lifetime are less likely to be present in such studies. Moreover, certain covariate values could be preferentially selected into the sample, being linked to the long-term survivors. The existing methodology for estimation of the propensity score using data collected on prevalent cases requires the correct conditional survival/hazard function given the treatment and covariates. This requirement can be alleviated if the disease under study has stationary incidence, the so-called stationarity assumption. We propose a nonparametric adjustment technique based on a weighted estimating equation for estimating the propensity score which does not require modeling the conditional survival/hazard function when the stationarity assumption holds. Large sample properties of the estimator is established and its small sample behavior is studied via simulation.


\end{abstract}

%  Please place your key words in alphabetical order, separated
%  by semicolons, with the first letter of the first word capitalized,
%  and a period at the end of the list.
%

%  As usual, the \maketitle command creates the title and author/affiliations
%  display

%  If you are using the referee option, a new page, numbered page 1, will
%  start after the summary and keywords.  The page numbers thus count the
%  number of pages of your manuscript in the preferred submission style.
%  Remember, ``Normally, regular papers exceeding 25 pages and Reader Reaction
%  papers exceeding 12 pages in (the preferred style) will be returned to
%  the authors without review. The page limit includes acknowledgements,
%  references, and appendices, but not tables and figures. The page count does
%  not include the title page and abstract. A maximum of six (6) tables or
%  figures combined is often required.''

%  You may now place the substance of your manuscript here.  Please use
%  the \section, \subsection, etc commands as described in the user guide.
%  Please use \label and \ref commands to cross-reference sections, equations,
%  tables, figures, etc.
%
%  Please DO NOT attempt to reformat the style of equation numbering!
%  For that matter, please do not attempt to redefine anything!

\section{Introduction}
\label{s:intro}

Survival or failure time data typically comprise an initiating event, say onset of a disease, and a terminating event, say death. In an ideal situation, recruited subjects have not experienced the initiating event, the so-called
\textit{incident} cases. These cases are then followed to a terminating event or censoring, say loss to follow-up. In many practical situations, however, recruiting incident cases is infeasible due to logistic or other constraints. In such circumstances, subjects who have experienced the initiating event prior to the start of the study, so-called \textit{prevalent} cases, are recruited. It is well known that these cases tend to have a longer survival time, and hence form a biased sample from the target population. This bias is termed \textit{length bias} when the initiating events are generated by a stationary Poisson process (\cite{cox1966statistical}, \cite{zelen1969theory}), the so-colled {\it{stationarity}} assumption.  %This bias in sampling can lead to bias in the estimation of the propensity score and an exposure effect of interest.

In observational studies, treatment is assigned to the experimental units without randomization. Thus, in each treatment group, the covariate distributions may be imbalanced which may lead to bias in estimating the treatment effect if the covariate imbalance is not properly taken into account (\cite{cochran1973controlling}, \cite{rubin1973use}). Propensity score is a tool that is widely used in causal inference to adjust for this source of bias (see \cite{robins2000marginal}, \cite{hernan2000marginal}). \cite{rosenbaum1983central} define the propensity score for a binary treatment $D$ as $p(D=1|\bX)$ where $\bX$ is a vector of measured covariates. They show that under some assumptions, treatment is independent of the covariates inside each propensity score stratum (the {\it balancing property} of the propensity score).

In cases where the sample is not representative of the population, naive propensity score estimation may not have the balancing property. \cite{cheng2012estimating} develop a method that consistently estimates the parameters of the propensity score from prevalent survival data. They also, present a method that can be used in a special case of length-biased sampling. Their method requires correct specification of the conditional hazard model given the treatment and covariates. We refer to their estimator as $CW$ in the sequel.

Our goal is to develop a method that estimates the propensity score using a weighted logistic regression where weights are estimated nonparametrically. Our estimating equation is designed specifically for the length-biased data, i.e., disease with stationary incidence. Unlike the method proposed by  CW, our method does not require any model specification for the conditional failure time given the exposure and the covariates. %We also generalize a nonparametric survival curve estimation method introduced by \cite{huang2011nonparametric} to accommodate the confounding effect as well as the length-biased sampling.

%This work is motivated by the Canadian Study of Health and Aging (CSHA).  In this study, individuals aged 65 or over were sample at random across Canada from communities and institutions for the elderly. Our sample includes individuals who  diagnosed with dementia  at the time of recruitment (prevalent cases). It is known that the CSHA data constitute a length-biased sample from the target population (see \cite{addona2006formal} and \cite{asgharian2006checking}). The scientific question of interest  is to compare  survival with dementia between two groups of patients, i.e., institutionalized, and those recruited from the community. To address this question, we focus on individuals whose date of institutionalization is before their onset of the decease. Since there are some covariates which confound the grouping effect on the survival time, the crude difference estimator will be biased. Hence, it is necessary to fit a propensity score first.

Studies on length-biased sampling can be traced as far back as
\cite{wicksell1925corpuscle} and his corpuscle problem. The phenomenon was
later noticed by \cite{fisher1934effect} in his article on methods of
ascertainment. \cite{neyman1955statistics} discussed length-biased sampling
further and coined the term {\it{incidence-prevalence}} bias.  \cite{cox1969some}
studied length-biased sampling in industrial applications,
while \cite{zelen1969theory} observed the same bias in screening tests for
disease prevalence (\cite{asgharian2002length}, \cite{2005}).  More recently, \cite{shen2009analyzing},
\cite{qin2010statistical} and \cite{ning2010non} have studied the analysis of covariates
under biased sampling.  

\section{Length-biased sampling}

\label{sec:l-bsamp}

In this section, we introduce concepts and notations necessary to formulate
problems involving length-biased sampling.  We adopt the common modeling
framework for prevalent cohort studies. We assume that affected individuals in
the study population develop the condition of interest (\textit{onset})
according to some stochastic mechanism at random points in calendar time, and
undergo a terminal event (\textit{failure}) at some subsequent time point that
is also determined by a stochastic mechanism. Individuals enter into the study
at some census time, and are followed up until the terminal or censoring event.

\subsection{Notations}

Let $T^{pop}$ be the time measured from the onset to failure time in the target
population with an absolutely continuous distribution $F$ and density $f$. Also, let  $D^{pop}$ and $\bX^{pop}$ be the binary treatment variable and the vector of covariates, respectively. Let $T$ be the same measured time
for \textbf{observed} subjects with distribution $F_{LB}$. The variables with superscript $pop$ represent the population variables; variables without $pop$ denote the observed truncated variables. It is well known
that if the onset times are generated by a stationary Poisson process, then
\begin{equation}
F_{LB}(t)= \frac{\displaystyle{\int_0^t} s \: dF(s)}{\displaystyle{\int_0^\infty} s \: dF(s)} = \frac{1}{\mu} \int_0^t s \: dF(s) \qquad \text{and} \qquad f_{LB}(t) = \frac{t f(t)}{\mu} \label{eq:LBdist}
\end{equation}
where $f_{LB}$ is the density function of $F_{LB}$ and $\mu$ is the mean survival time under $F$. The observed event time, $T$, can be written $A+R$, when $A$ is the time from the onset of the disease to the recruitment time, and $R$, the residual life time, is the time from recruitment to the event, also called backward and forward recurrence times, respectively. When the individuals are also subject to right-censoring, the observed survival time is $Y=A+\min(R,C)$, where $C$ is the censoring time measured from the recruitment to the loss to follow-up; for all subjects, both $A$ and $\min(R,C)$ are observed.  The censoring indicator is denoted by $\delta$ ($\delta=1$ indicating failure).  The sample consists of $(y_i, a_i, \delta_i, d_i, \bx_i)$ for $n$ independent subjects. The following diagram illustrates the different random quantities introduced in this Section.

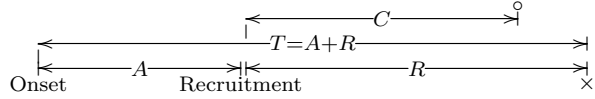
\vspace{.1in}
\centerline{
\xymatrix@1{
  \ar@{|<->|}@<1ex>[rrrrrrrr] | -{T=A+R} &  &  &   \ar@{|<->|}@<1ex>[lll] | -{A}^ >{\text{Onset}} ^ <{\text{Recruitment}}  \ar@{|<->|}@<3ex>[rrrr] | -{C} ^>{ \circ}&   &  &  &    &   \ar@{|<->|}@<1ex>[lllll] | -{R} ^ <{\times}\\
   }
}
\vspace{.1in}

Throughout the paper, we assume that the following assumptions  hold:
\begin{itemize}
\item[] {\it{A1.}} The variable $(T^{pop},D^{pop},\bX^{pop})$ is independent of the calendar time of the onset of the disease. 
\item[] {\it{A2.}} The disease has stationary incidence, i.e., the disease incidence occurs at a constant rate.
\item[] {\it{A3.}} The censoring time $C$ is independent of $(A,R)$ given the covariates $(D,\bX)$.
\item[] {\it{A4.}} The censoring time $C$ is independent of the covariates $(D,\bX)$.
\end{itemize}

Let $f(t|d,\bx)$ be the unbiased conditional density of survival times given the covariates and treatment. Then, under assumptions $A1$ and $A2$, the joint density of  $(A,T)$ given $(D,\bX)$ is $f(t|d,\bx)/\int_0^\infty u f(u|d,\bx) du  I(t>a>0)$ as shown in \cite{asgharian2006checking}. The assumptions $A3$ and $A4$ is used to show that
\[
p(Y \in (t,t+dt), A \in (a,a+da), \delta=1 |d,\bx)=\frac{f(t | d,\bx) S_c(t-a) dtda}{\mu_d(\bx,\theta)},
\]
where $\mu_d(\bx,\theta)=\int_0^\infty p(\Tpop\geq a|D=d,\bX=\bx,\theta)$ and $S_c(.)$ are the conditional counterfactual mean failure time if treated at $D = d$ and the survival function for the residual censoring variable $C$, respectively. By integrating the above equation over $0<a<t$, we have 
\begin{align}
p(Y \in (t,t+dt), \delta=1 |d,\bx)=\frac{f(t | d,\bx) w(t) dt}{\mu_d(\bx,\theta)},
\label{eq:shen1}
\end{align}
where $w(t) = \int_0^t S_C(s) \: ds$ (see \cite{shen2009analyzing} and \cite{qin2010statistical}).

\section{Propensity score estimation  under length-biased sampling}
\label{se:PSest}
Assuming a logit model for the propensity score in the target population, we have
\begin{align}
\pi(\bx,\alpha)=p(D^{pop}=1|\bX^{pop}=\bx) = \frac{\exp( \alpha \bx)}{1+\exp( \alpha \bx)}. \label{eq:unbps}
\end{align}
where $\alpha$ is a $p\times1$ vector of parameters. The vector of covariates $X$ may include a column of 1s. It can be shown that  under assumption $A2$, we have
\begin{align}
    p_{LB}(D=1|\bX=\bx) & =  \frac{\mu_1(\bx,\theta)p(D^{pop}=1|\bX^{pop}=\bx)}{\mu(\bx,\alpha,\theta)}, \label{eq:dlb}
\end{align}
where $\mu(\bx,\alpha, \theta)= \pi(\bx,\alpha)\mu_1(\bx,\theta)+(1-\pi(\bx,\alpha))\mu_0(\bx,\theta)$ (\cite{bergeron2008covariate}).

 Assuming the proportional hazard model, i.e., $\lambda_{T^{pop}}(u|D^{pop}=d, \bX^{pop}=\bx)=\lambda_0(u) e^{\gamma d + \beta \bx}$, \cite{cheng2012estimating} show that the parameter of the propensity score can be consistently estimated using the logistic regression but adjusted for the `offset' term $\log(\hat \alpha(x;\hat \Lambda,\hat \gamma,\hat \beta))$ as the intercept where $\hat \alpha(x;\hat \Lambda,\hat \gamma,\hat \beta) = \frac{\sum_{i=1}^n \exp[-\hat \Lambda(a_i) \exp(\hat \gamma  + \hat \beta \bx)}{ \sum_{i=1}^n \exp[-\hat \Lambda(a_i) \exp( \hat \beta \bx)]}$.
The cumulative hazard function  $\Lambda$ is estimated using the Breslow estimator, and the parameters $(\gamma,\beta)$ can be estimated using   the estimating equations developed by \cite{qin2010statistical}. The consistency of the parameters of the propensity score in the CW method relies on the correct  specification of the conditional hazard model given the treatment and covariates.

%In causal inference,  when the marginal causal effect is the parameter of interest, investigators may prefer not to  specify the conditional hazard model. So it seems useful to develop a method that can account for the length-biased sampling without requiring any  assumptions on the conditional hazard model.

When the initiating event of the duration variable has stationary incidence, it is possible to devise a robust methodology for estimating the propensity score that does not require knowledge of the conditional hazard model. See among others \cite{wolfson2001reevaluation} and \cite{de2004nonparametric} for examples of such duration variables in medical and labor force studies, respectively.

 We construct an unbiased
estimating equation for estimating the parameters of the propensity score using the weighted logistic regression where weights are estimated nonparametrically. %Since unbiasedness of the estimating equation  does not require any  assumptions on the conditional hazard model, it is particularly of interest when the marginal causal effect is the parameter of interest. 
Let  $F(d|\bx)$ be the unbiased conditional distribution of the treatment given the covariates.
%then {\small{
%\begin{eqnarray*}
%\E \left[\delta \frac{(D-\pi(\bx,\alpha))}{w(Y)}\:\bigg\vert\:\bX \right]&=&
%\frac{1}{\mu(D,\bX)}\int \int \frac{(D-\pi(\bx,\alpha))}{w(y)} f_U(y|\bX,D) w(y) \: dy \: dF_U(d|\bX) \\[6pt]
%&=& \frac{1}{\mu(D,\bX)}\int \int (D-\pi(\bx,\alpha)) f_U(y|\bX,D) \: dy \: dF_U(d|\bX)\\[6pt]
%& = & 0,
%\end{eqnarray*} }}
%then {\small{
%\begin{eqnarray*}
%\E \left[\delta \frac{(D-\pi(\bx,\alpha))}{w(Y)}\:\bigg\vert\:\bX \right]&=& \E \left[\E\left\{\delta \frac{(D-\pi(\bx,\alpha))}{w(Y)}|\bD,\bX\right\}\right] \\
%&=&\frac{1}{\mu(\theta,\bX)}\E \left[ D-\pi(\bx,\alpha) |\bX \right] \int \frac{1}{w(y)} f_U(y|\bX,D) w(y) \: dy\\[6pt]
%&=& \frac{1}{\mu(\theta,\bX)}\E \left[D-\pi(\bx,\alpha) |\bX \right]\\[6pt]
%& = & 0,
%\end{eqnarray*} }}
Then
\begin{eqnarray*}
\E \left[\delta \frac{(D-\pi(\bx,\alpha))}{w(Y)}\:\bigg\vert\:\bX=\bx \right]&=& \E \left[\E\left\{\delta \frac{(d-\pi(\bx,\alpha))}{w(Y)}\:\bigg\vert\:D=d,\bX = \bx \right\}\right] \\
&=& \int (d-\pi(\bx,\alpha)) \int \frac{f(y|\bx,d) w(y)}{w(y)\mu_d(\bx,\theta)}   \: dy \\
& &\hspace{1.5in} \times\: \frac{\mu_d(\bx,\theta)dF(d|\bx)}{\mu(\bx,\alpha,\theta)}\\[6pt]
&=& \frac{1}{\mu(\bx,\alpha,\theta)}\int (d-\pi(\bx,\alpha)) \: dF(d|\bx) = 0.
\end{eqnarray*}
The second equality follows from equations (\ref{eq:shen1}), and (\ref{eq:dlb}).
The last equality holds since $f(y|\bx,d)$ is a proper density and (\ref{eq:unbps}). An unbiased estimating equation for $\alpha$ is therefore
\begin{align}
U(\alpha)=\sum_{i=1}^n \delta_i \bx_i^\top \frac{(d_i-\pi(\bx_i,\alpha))}{\widehat w(y_i)}=0,
\label{eq:PSEE}
\end{align}
where $\hat{w}(y) = \int_{0}^{y} \hat{S}_{C}(s) ds$ and $\hat{S}_{C}$ is the
Kaplan-Meier estimator of the survivor function of the residual censoring
variable $C$.
%This is equivalent to the weighted logistic regression among the uncensored subjects.

%Note that, as censoring is presumed a mechanism acting independently of the all %other random quantities, the censoring

 The following theorem presents the asymptotic properties of the estimators obtained by (\ref{eq:PSEE}) in the presence of
length-biased sampling  when $w(.)$ is replaced by its estimated value.

%%%%%%%%%%%%%%%%%%%%%%%%%%%%%%%%%%%%%%%%%%%%%%%%%%%%%%%%%%THEOREM%%%%%%%%%%%%%%%%%%%%%%%%%%%%%%%%%%%%%%%%%%%%%%%%%%%%%%%%%%%%%%%%%%%%
\begin{theorem}
Let  $\widehat \alpha$ be an estimator obtained by (\ref{eq:PSEE}). Then under  conditions $C1-C6$ and assumptions $A1-A4$,   $\widehat \alpha \rightarrow \alpha_0$ in probability as $n \rightarrow \infty$. Moreover,
\[
\sqrt{n}(\widehat \alpha-\alpha) \stackrel{d}{\longrightarrow}\Normal(0, \eta({\alpha})),
\]
\label{th:PSas}
where $\eta(\alpha)$ is given in the Appendix.
\end{theorem}

\begin{proof} See the Appendix.  \end{proof}
A consistent plug-in estimator of $\eta(\alpha)$ is presented in the Appendix.

\section{Simulation Studies}
\label{sec:sim}
In this Section, we describe a simulation study to examine the performance of
the proposed propensity score estimator. Our simulation consists of 500 datasets of sizes 500 and 5000. The censoring variable $C$ is generated from a uniform distribution in the interval $(0,\tau)$ where the parameter $\tau$ is set such that it results in a desired censoring proportion. To create length-biased samples, we generate a variable $A$ from a uniform distribution $(0,\rho)$ and ignore those whose generated unbiased failure time is less than $A$.

 We generated the unbiased failure times from the following hazard model  $h(t|d,\bx)=0.2 \exp\{d-0.5 x_1+0.5x_2+0.5dx_1-0.5dx_2\},$ where $
D \sim \textsf{Bernoulli} \left(\frac{\exp\{-0.1+1x_1-1x_2\}}{1+\exp\{-0.1+1x_1-1x_2\}}\right)$ with $X_1$ and $X_2$ distributed according to $N(0,\sigma=0.5)$. We estimate the parameters of the propensity score using CW and the proposed method and compare the results with the true values. We assume three different censoring proportions 10\%, 20\% and 30\%.

\begin{table*}
\centering
 \def\~{\hphantom{0}}
 \begin{minipage}{130mm}
\caption{\label{tab:CoxATE} Simulation: Propensity score parameter estimation. $\hat \alpha$: Estimated parameters using the proposed method. $\hat \alpha_w$: Estimated parameters using the CW method. $\hat \alpha_w^m$: Estimated parameters using the CW method when the hazard model is misspecified. $\hat \alpha_{Un}$: Estimated parameters when unadjusted for the length-biased sampling. $\alpha$=(-0.1,1,-1).}
\begin{tabular}{*{3}{c} |*{2}{c}} \hline
Method & Bias &  S.D.  & Bias &S.D.  \\ \hline
10 \% Cens.   & \multicolumn{2}{c}{$n=500$} & \multicolumn{2}{c}{$n=5000$}\\
$\hat \alpha$ &(0.01,0.04,0.03) &(0.21,0.43,0.45) &(0.00,0.01,0.01) & (0.09,0.18,0.19) \\
$\hat \alpha_w$ &(0.08,0.02,0.01) &(0.17,0.28,0.28) &(0.01,0.00,0.01)  &  (0.06,0.10,0.10)   \\
$\hat \alpha_w^m$ &(0.03,0.02,0.49) &(0.17,0.29,0.23) &(0.08,0.03,0.48) &  (0.06,0.09,0.08)   \\
$\hat \alpha_{Un}$ &(0.10,0.50,0.50) &(0.11,0.21,0.22) &(0.10,0.51,0.50) &  (0.04,0.07,0.08)   \\ \hline
20 \% Cens.   & \multicolumn{2}{c}{$n=500$} & \multicolumn{2}{c}{$n=5000$}\\
$\hat \alpha$ &(0.01,0.05,0.05) &(0.22,0.42,0.43)& (0.02,0.03,0.03)& (0.09,0.18,0.20) \\
$\hat \alpha_w$ &(0.02,0.01,0.01) &(0.17,0.29,0.27) &(0.02,0.01,0.01) &  (0.06,0.10,0.10)   \\
$\hat \alpha_w^m$ &(0.01,0.02,0.47) &(0.16,0.29,0.21) &(0.02,0.05,0.46) &  (0.06,0.10,0.08)   \\
$\hat \alpha_{Un}$ &(0.11,0.49,0.50)  &(0.10,0.22,0.20) &(0.10,0.51,0.50)&  (0.04,0.07,0.08)   \\ \hline
30 \% Cens.  & \multicolumn{2}{c}{$n=500$} & \multicolumn{2}{c}{$n=5000$}\\
$\hat \alpha$ &(0.04,0.08,0.09) &(0.22,0.44,0.44)& (0.03,0.06,0.07)& (0.10,0.19,0.20) \\
$\hat \alpha_w$ &(0.07,0.00,0.02) &(0.17,0.29,0.29) &(0.08,0.03,0.02) &  (0.06,0.10,0.11)   \\
$\hat \alpha_w^m$ &(0.07,0.02,0.45)  &(0.17,0.29,0.21) &(0.03,0.06,0.45) &  (0.06,0.10,0.08)   \\
$\hat \alpha_{Un}$ &(0.11,0.49,0.50)  &(0.10,0.22,0.20) &(0.10,0.51,0.50)&  (0.04,0.07,0.08)   \\ \hline
\end{tabular}
%\caption*{{ Unadjusted$^{lc}$: neither the length-biased nor the confounding are adjusted for.  Unadjusted$^c$: The length-biased is adjusted whereas confounding left unadjusted. Unadjusted$^l$: The confounding is adjusted whereas the length-biased left unadjusted. }}
\end{minipage}
\vspace*{+16pt}
\end{table*}

We estimate the parameters of the propensity score using  four different estimators: $\hat \alpha$ is the estimator obtained by the proposed method; $\hat \alpha_w$ and  $\hat \alpha_w^m$ are the estimator obtained by the CW method when the hazard model is correctly and incorrectly specified, respectively, and $\hat \alpha_{Un}$ is obtained by a naive method that does not adjust for the length-bias sampling. In $\hat \alpha_w^m$, we assume that the interaction between the treatment $D$ and the covariate $X_2$ has been ignored in the fitted hazard model.

Table \ref{tab:CoxATE} summarizes the  estimated propensity score parameters and their standard errors.  Our simulation results confirm that the proposed estimating equation (\ref{eq:PSEE}) adjusts the length-biased sampling. The standard errors, however, are larger than the one obtained by  the CW method, which is the price we pay for relaxing the modeling assumption of the hazard model. As we expected, CW estimator is highly sensitive to the model misspecification even when just one of the interaction terms is ignored. Specifically,  when the interaction term between the treatment and variable $X_2$ is omitted in the fitted hazard model, the estimated coefficient corresponding to $X_2$ in the propensity score model is biased. In general, if variables in the study are correlated, then missing one variable in the hazard model may cause bias in the estimation of other variables in the propensity score model as well.

\section*{Acknowledgements}
This work was supported in part by NIDA grant P50 DA010075.

%  If your paper refers to supplementary web material, then you MUST
%  include this section!!  See Instructions for Authors at the journal
%  website http://www.biometrics.tibs.org

%\section*{Supplementary Materials}

%Web Appendix A, referenced in Section~\ref{s:model}, is available with
%this paper at the Biometrics website on Wiley Online
%Library.\vspace*{-8pt}

%  Here, we create the bibliographic entries manually, following the
%  journal style.  If you use this method or use natbib, PLEASE PAY
%  CAREFUL ATTENTION TO THE BIBLIOGRAPHIC STYLE IN A RECENT ISSUE OF
%  THE JOURNAL AND FOLLOW IT!  Failure to follow stylistic conventions
%  just lengthens the time spend copyediting your paper and hence its
%  position in the publication queue should it be accepted.

%  We greatly prefer that you incorporate the references for your
%  article into the body of the article as we have done here
%  (you can use natbib or not as you choose) than use BiBTeX,
%  so that your article is self-contained in one file.
%  If you do use BiBTeX, please use the .bst file that comes with
%  the distribution.

\bibliographystyle{Biometrika}
\bibliography{mybibDB}

\begin{thebibliography}{23}
\expandafter\ifx\csname natexlab\endcsname\relax\def\natexlab#1{#1}\fi
\expandafter\ifx\csname url\endcsname\relax
  \def\url#1{\texttt{#1}}\fi
\expandafter\ifx\csname urlprefix\endcsname\relax\def\urlprefix{URL }\fi
\providecommand{\eprint}[2][]{\url{#2}}

\bibitem[{Asgharian et~al.(2002)Asgharian, M'Lan \&
  Wolfson}]{asgharian2002length}
\textsc{Asgharian, M.}, \textsc{M'Lan, C.~E.} \& \textsc{Wolfson, D.~B.}
  (2002).
\newblock Length-biased sampling with right censoring.
\newblock \textit{Journal of the American Statistical Association} 97 201--209.

\bibitem[{Asgharian \& Wolfson(2005)}]{2005}
\textsc{Asgharian, M.} \& \textsc{Wolfson, D.~B.} (2005).
\newblock Asymptotic behavior of the unconditional {NPMLE} of the length-biased
  survivor function from right censored prevalent cohort data.
\newblock \textit{The Annals of Statistics} 33 pp. 2109--2131.

\bibitem[{Asgharian et~al.(2006)Asgharian, Wolfson \&
  Zhang}]{asgharian2006checking}
\textsc{Asgharian, M.}, \textsc{Wolfson, D.~B.} \& \textsc{Zhang, X.} (2006).
\newblock Checking stationarity of the incidence rate using prevalent cohort
  survival data.
\newblock \textit{Statistics in medicine} 25 1751--1767.

\bibitem[{Bergeron et~al.(2008)Bergeron, Asgharian \&
  Wolfson}]{bergeron2008covariate}
\textsc{Bergeron, P.~J.}, \textsc{Asgharian, M.} \& \textsc{Wolfson, D.~B.}
  (2008).
\newblock Covariate bias induced by length-biased sampling of failure times.
\newblock \textit{Journal of the American Statistical Association} 103
  737--742.

\bibitem[{Cheng \& Wang(2012)}]{cheng2012estimating}
\textsc{Cheng, Y.} \& \textsc{Wang, M.} (2012).
\newblock Estimating propensity scores and causal survival functions using
  prevalent survival data.
\newblock \textit{Biometrics} 68 707--716.

\bibitem[{Cochran \& Rubin(1973)}]{cochran1973controlling}
\textsc{Cochran, W.} \& \textsc{Rubin, D.} (1973).
\newblock Controlling bias in observational studies: A review.
\newblock \textit{Sankhy{\=a}: The Indian Journal of Statistics, Series A}
  417--446.

\bibitem[{Cox(1969)}]{cox1969some}
\textsc{Cox, D.~R.} (1969).
\newblock \textit{Some Sampling Problems in Technology}.
\newblock in New Developments in Survey Sampling: New York: Wiley Interscience,
  506--527.

\bibitem[{Cox \& Lewis(1966)}]{cox1966statistical}
\textsc{Cox, D.~R.} \& \textsc{Lewis, P.} (1966).
\newblock \textit{The statistical analysis of series of events}.
\newblock John Wiley and Sons.

\bibitem[{De~U{\~n}a-{\'a}lvarez(2004)}]{de2004nonparametric}
\textsc{De~U{\~n}a-{\'a}lvarez, J.} (2004).
\newblock Nonparametric estimation under length-biased sampling and type i
  censoring: a moment based approach.
\newblock \textit{Annals of the Institute of Statistical Mathematics} 56
  667--681.

\bibitem[{Fisher(1934)}]{fisher1934effect}
\textsc{Fisher, R.~A.} (1934).
\newblock The effect of methods of ascertainment upon the estimation of
  frequencies.
\newblock \textit{Annals of Human Genetics} 6 13--25.

\bibitem[{Hern{\'a}n et~al.(2000)Hern{\'a}n, Brumback \&
  Robins}]{hernan2000marginal}
\textsc{Hern{\'a}n, M.}, \textsc{Brumback, B.} \& \textsc{Robins, J.} (2000).
\newblock Marginal structural models to estimate the causal effect of
  zidovudine on the survival of hiv-positive men.
\newblock \textit{Epidemiology} 11 561--570.

\bibitem[{Neyman(1955)}]{neyman1955statistics}
\textsc{Neyman, J.} (1955).
\newblock Statistics--servant of all science.
\newblock \textit{Science} 122 401--406.

\bibitem[{Ning et~al.(2010)Ning, Qin \& Shen}]{ning2010non}
\textsc{Ning, J.}, \textsc{Qin, J.} \& \textsc{Shen, Y.} (2010).
\newblock Non-parametric tests for right-censored data with biased sampling.
\newblock \textit{Journal of the Royal Statistical Society: Series B
  (Statistical Methodology)} 72 609--630.

\bibitem[{Pepe \& Fleming(1991)}]{pepe1991weighted}
\textsc{Pepe, M.~S.} \& \textsc{Fleming, T.~R.} (1991).
\newblock Weighted {K}aplan-{M}eier statistics: Large sample and optimality
  considerations.
\newblock \textit{Journal of the Royal Statistical Society. Series B
  (Methodological)} 53 341--352.

\bibitem[{Qin \& Shen(2010)}]{qin2010statistical}
\textsc{Qin, J.} \& \textsc{Shen, Y.} (2010).
\newblock Statistical methods for analyzing right-censored length-biased data
  under cox model.
\newblock \textit{Biometrics} 66 382--392.

\bibitem[{Robins et~al.(2000)Robins, Hern{\'a}n \&
  Brumback}]{robins2000marginal}
\textsc{Robins, J.}, \textsc{Hern{\'a}n, M.} \& \textsc{Brumback, B.} (2000).
\newblock Marginal structural models and causal inference in epidemiology.
\newblock \textit{Epidemiology} 11 550--560.

\bibitem[{Rosenbaum \& Rubin(1983)}]{rosenbaum1983central}
\textsc{Rosenbaum, P.~R.} \& \textsc{Rubin, D.~B.} (1983).
\newblock The central role of the propensity score in observational studies for
  causal effects.
\newblock \textit{Biometrika} 70 41--55.

\bibitem[{Rubin(1973)}]{rubin1973use}
\textsc{Rubin, D.} (1973).
\newblock The use of matched sampling and regression adjustment to remove bias
  in observational studies.
\newblock \textit{Biometrics}  185--203.

\bibitem[{Shen et~al.(2009)Shen, Ning \& Qin}]{shen2009analyzing}
\textsc{Shen, Y.}, \textsc{Ning, J.} \& \textsc{Qin, J.} (2009).
\newblock Analyzing length-biased data with semiparametric transformation and
  accelerated failure time models.
\newblock \textit{Journal of the American Statistical Association} 104
  1192--1202.

\bibitem[{Wang(1991)}]{wang1991nonparametric}
\textsc{Wang, M.~C.} (1991).
\newblock Nonparametric estimation from cross-sectional survival data.
\newblock \textit{Journal of the American Statistical Association} 86 130--143.

\bibitem[{Wicksell(1925)}]{wicksell1925corpuscle}
\textsc{Wicksell, S.~D.} (1925).
\newblock The corpuscle problem: a mathematical study of a biometric problem.
\newblock \textit{Biometrika} 17 84--99.

\bibitem[{Wolfson et~al.(2001)Wolfson, Wolfson, Asgharian, M'Lan, {\O}stbye,
  Rockwood \& Hogan}]{wolfson2001reevaluation}
\textsc{Wolfson, C.}, \textsc{Wolfson, D.~B.}, \textsc{Asgharian, M.},
  \textsc{M'Lan, C.~E.}, \textsc{{\O}stbye, T.}, \textsc{Rockwood, K.} \&
  \textsc{Hogan, D.~B.} (2001).
\newblock A reevaluation of the duration of survival after the onset of
  dementia.
\newblock \textit{New England Journal of Medicine} 344 1111--1116.

\bibitem[{Zelen \& Feinlein(1969)}]{zelen1969theory}
\textsc{Zelen, M.} \& \textsc{Feinlein, M.} (1969).
\newblock On the theory of screening for chronic diseases.
\newblock \textit{Biometrika} 56 601--614.

\end{thebibliography}

\appendix

%  To get the journal style of heading for an appendix, mimic the following.

\section*{Appendix}

In this section, we present the assumptions and proofs of the main result. The following conditions are required for establishing Theorem  \ref{th:PSas}:
\begin{enumerate}
%\item[C.1] $\pi(\bX,\alpha) $ is a continuously differentiable function.
\item[C.1] $\bX$ is a  $p \time 1$ vector of bounded covariates, not contained in a $(p-1)$ dimensional hyperplane.
\item[C.2] $\sup[t: p(R>t)>0] \geq \sup[t: p(C>t)>0]$=$s$ and $p(\delta=1)>0$.

\item[C.3] $\int_0^s[(\int_t^s S_C(v)dv)^2/(S^2_C(t)S_V(t))]dS_C(t)<\infty$.

\item[C.4] det $\E \left[  \left\{ \delta \bX^\top \frac{D-\pi(\bX,\alpha)}{ w(Y)}\right\}^{\otimes 2} \right]<\infty$.

\item[C.5] $\Lambda= \E \left[ \delta \bX^\top \frac{\partial \pi(\bX,\alpha)}{\partial \alpha}  \right]$ is nonsingular.

\item[C.6] det [$\int_0^s \kappa_1^2(t)/(S_C^2(t)S_R(t))dS_C(t)]<\infty$ where $\kappa_1(t)=\E\left[ \frac{ \delta \I(Y>t) \bX^\top [D -\pi(\bX,\alpha))]\int_t^Y S_C(v)dv}{ w^2(Y)}\right]$, 
\end{enumerate}
C.2 is an identifiability condition (\cite{wang1991nonparametric}), C.3-C.6 are required to obtain an estimator
with a finite variance.

\bigskip

%%%%%%%%%%%%%%%%%%%%%%%%%%%%%%%%%%%%%%%%%%%%%%%%%%%%%%%%%%%%%By the no unmeasured confounder assumption and the independence of $D$ and $(\delta=1,Y(1),Y(0))$ given $\bX$, we have%%%%%%%%%%%%%%%%%%%%%%%%%%%%%%%%%%%%%%%%%%%%%%%%%%%%%%%%%%%%%%%THEOREM%%%%%%%%%%%%%%%%%%%%%%%%%%%%%%%%%%%%%%
\noindent \textbf{\it{Proof of Theorem \ref{th:PSas}}}. Using the strong consistency of $\widehat w(y)$ to $w(y)$
(\cite{pepe1991weighted}), we have
\begin{align*}
\frac{1}{\widehat w(Y)} = \frac{1}{ w(Y)} \left[  1+\frac{w(Y)-\widehat w(Y)}{w(Y)}  \right]+o_p(1).
%\frac{\delta_j \I(Y_j\geq u)}{\widehat w(Y_j)}&=
%\left[\frac{\delta_j \I(Y_j\geq u)}{ w(Y_j)}\right] \left[  1+\frac{w(Y_j)-\widehat w(Y_j)}{w(Y_j)}  \right]+o_p(1),
\end{align*}
Thus
\begin{align}
\tilde U(\alpha)&=\frac{1}{\sqrt n}\sum_{i=1}^n U_i(\alpha) \nonumber\\
&= \frac{1}{\sqrt n}\sum_{i=1}^n \delta_i \bx_i^\top \frac{d_i-\pi(\bx_i,\alpha)}{ \hat w(y_i)} =\frac{1}{\sqrt n}  \sum_{i=1}^n \delta_i \bx_i^\top \frac{d_i-\pi(\bx_i,\alpha)}{ w(y_i)} \left[  1+\frac{w(y_i)-\widehat w(y_i)}{w(y_i)}  \right]+o_p(1) ,
\label{eq:utilde}
\end{align}
In section \ref{se:PSest}, we have shown that the estimating equation $ \tilde U(\alpha)$ given by (\ref{eq:PSEE}) is unbiased. Following the martingale integral representation
$\sqrt n (\widehat w(Y)-w(Y))$ (\cite{shen2009analyzing} and \cite{qin2010statistical}),
\begin{align}
 \frac{w(y_i)-\widehat w(y_i)}{w(y_i)}= \E \left[  \int_0^s \frac{\I(y_i>t)\int_t^{y_i} S_C(v)dv}{ S_C(t)S_R(t) w(y_i)}dM_C(t)    \right],
 \label{eq:pep}
\end{align}
where $M_C(s)=\I(Y-A<s,\delta=0)-\int_0^s \I(\min(Y-A,C)>u) \: d\Lambda_C(u)$, with $\Lambda_C(.)$ be the cumulative hazard function of the censoring variable. The stochastic process $M_C(s)$ has mean zero,
\begin{align*}
\E[M_C(s)] &= \E[\I(C<Y-A<s )] -  \int_0^s  \E[ \I(Y-A>u). \I(C>u) ]\: d\Lambda_C(u) \\
             &= \int_0^s S_C(u) \lambda_C(u) S_R(u) du - \int_0^s S_C(u) S_R(u) d\Lambda_C(u)=0.
\end{align*}

By inserting (\ref{eq:pep}) into (\ref{eq:utilde}) and using the standard Taylor expansion, we can derive the asymptotic variance of the estimator as follows
\[
\eta(\alpha) = \Lambda' \Sigma^{-1} \Lambda,
\]
where
\begin{align*}
\Sigma=\E\left[\tilde U(\alpha) \tilde U(\alpha)Õ\right]&=\E \left[  \left\{ \delta \bX^\top \frac{D-\pi(\bX,\alpha)}{ w(Y)}\right\}^{\otimes 2} \left\{  1+\frac{w(Y)-\widehat w(Y)}{w(Y)} \right\}^{2}  \right] \\
     &=\E \left[  \left\{ \delta \bX^\top \frac{D-\pi(\bX,\alpha)}{ w(Y)}+\int_0^{s} \frac{\kappa_1(t)dM_C(t)}{S_C(t)S_R(t)} \right\}^{\otimes 2}  \right] \\
\Lambda&= \E \left[ \frac{\partial  U_i(\alpha)}{\partial \alpha}  \right] = \E \left[ \frac{\delta}{w(Y)} \bX^\top \frac{\partial \pi(\bX,\alpha)}{\partial \alpha} \right]
\end{align*}
where $\kappa_1(t)=\E\left[ \frac{ \delta \I(Y>t) \bX^\top [D -\pi(\bX,\alpha))]\int_t^Y S_C(v)dv}{ w^2(Y)}\right]$. Let $\P_n$ be the empirical average. 
The components of the variance-covariance matrix $\eta(\alpha)$ can be consistently estimated by
\begin{align*}
 \hat \Sigma    &=\P_n \left[  \left\{ \delta \bX^\top \frac{D-\pi(\bX,\alpha)}{ \hat w(Y)}+\int_0^{s} \frac{\hat \kappa_1(t)d\hat M_C(t)}{\hat S_C(t)\hat S_R(t)} \right\}^{\otimes 2}  \right], \\
\hat \Lambda&= \P_n \left[ \frac{\partial  U_i(\alpha)}{\partial \alpha}  \right] = \P_n \left[ \frac{\delta}{\hat w(Y)} \bX^\top \frac{\partial \pi(\bX,\alpha)}{\partial \alpha} \right],
\end{align*}
Also, the stochastic process $M_C(s)$ can be estimated by replacing the $\Lambda_C(.)$ by its
estimate, $\widehat \Lambda_C(.)$.

\label{lastpage}

\end{document}